\documentclass[11pt]{article}
\usepackage{amsmath, amssymb, amsthm, mathrsfs, mathtools, amscd}
\usepackage[latin1]{inputenc}
\usepackage{hyperref}
\usepackage[all]{xy}

\newtheorem{theorem}{Theorem}[section]
\newtheorem{lemma}[theorem]{Lemma}
\newtheorem{corollary}[theorem]{Corollary}
\newtheorem{proposition}[theorem]{Proposition}
\newtheorem{definition}[theorem]{Definition}

\usepackage{color}
\definecolor{darkgreen}{rgb}{0,.66,0}
\usepackage[normalem]{ulem} 

\newcommand\on[1]{\operatorname{#1}}
\newcommand\mc[1]{\mathcal{#1}}
\newcommand\ps[1]{\underline{#1}}

\newcommand\ld{\lambda}

\newcommand{\sa}{\on{sa}}
\newcommand{\op}{\on{op}}              

\newcommand{\hA}{\hat{A}}

\newcommand\eq[1]{(\ref{#1})}

\newcommand{\Sig}{\ps{\Sigma}}            

\newcommand\cN{\mc{N}}
\newcommand\cM{\mc{M}}

\newcommand\VN{\mc{V}(\cN)}
\newcommand\VM{\mc{V}(\cM)}

\newcommand\bbC{\mathbb{C}}
\newcommand\bbR{\mathbb{R}}
\newcommand\PN{\mc{P}(\cN)}
\newcommand\PM{\mc{P}(\cM)}

\newcommand\cA{\mc{A}}
\newcommand\cB{\mc{B}}
\newcommand\Ob[1]{\on{Ob}(#1)}

\newcommand\ra{\rightarrow}
\newcommand\mt{\mapsto}
\newcommand\lra{\longrightarrow}
\newcommand\lmt{\longmapsto}
\newcommand\hra{\hookrightarrow}

\newcommand\tphi{\tilde\phi}

\newcommand\SigB{\Sig^\cB}
\newcommand\SigN{\Sig^{\cN}}
\newcommand\SigM{\Sig^{\cM}}
\newcommand\cG{\mc G}

\newcommand\Aut{\on{Aut}}
\newcommand\ol[1]{\overline{#1}}
\newcommand\de{\delta}
\newcommand\io{\iota}
\newcommand\cJ{\mc J}

\newcommand\Ga{\Gamma}
\newcommand\pair[2]{\langle #1,#2\rangle}

\newcommand\cD{\mc D}
\newcommand\Pre[1]{\mathbf{Presh}(#1)}

\newcommand\KHaus{\mathbf{KHaus}}

\newcommand\JBW{\mathbf{JBW}}
\newcommand\JB{\mathbf{JB}}
\newcommand\tF{\tilde F}
\newcommand\Hom{\on{Hom}}

\newcommand\HStone{\mathbf{HStone}}

\begin{document}

\title{\textbf{Two New Complete Invariants\\of von Neumann Algebras}}

\author{Andreas D\"oring}

\date{19. November 2014}

\maketitle

\begin{abstract}
We show that the oriented context category and the oriented spectral presheaf are complete invariants of a von Neumann algebra not isomorphic to $\bbC\oplus\bbC$ and with no direct summand of type $I_2$.
\end{abstract}

\section{Introduction}
\subsection{Background and motivation}
The context category $\VM$ of a von Neumann algebra $\cM$ is the set of commutative von Neumann subalgebras of $\cM$, partially ordered by inclusion. The spectral presheaf $\SigM$ of a von Neumann algebra $\cM$ is a presheaf over $\VM$ and serves as a generalisation of the Gelfand spectrum to arbitrary von Neumann algebras.\footnote{Definitions are given in the next subsection.}. The spectral presheaf plays a key role in the topos approach to quantum theory \cite{DoeIsh11} Isham and Butterfield originally introduced the context category and the spectral presheaf in \cite{IshBut98} and, for von Neumann algebras, in $\cite{IHB00}$; the latter article is jointly with Hamilton. 

From a mathematical perspective, it is natural to ask whether the context category or the spectral presheaf are complete invariants of a von Neumann algebra. A first partial answer was given in \cite{DoeHar10}: let $\cM,\cN$ be von Neumann algebras not isomorphic to $\bbC\oplus\bbC$ and with no direct summands of type $I_2$. There is a bijective correspondence between order-isomorphisms $\tilde f:\VM\ra\VN$ of the context categories and Jordan $*$-isomorphisms $f:\cM\ra\cN$. Hence, the mere order structure on commutative subalgebras suffices to fix the Jordan algebraic structure of a von Neumann algebra. Subsequently, Hamhalter considered the case of $C^*$-algebras in \cite{Ham11}. In general, the context category of a $C^*$-algebra determines the algebra up to quasi-Jordan isomorphism. For certain $C^*$-algebras for which a generalised version of Gleason's theorem due to Bunce and Wright \cite{BunWri95} applies (showing that certain quasi-linear maps between quasi-Jordan algebras are linear), one recovers the Jordan structure.

In this article, we will only consider von Neumann algebras. Concerning these, the next step was taken in \cite{Doe12b}, where it was shown that there is a bijective correspondence between order-isomorphisms $\tilde f:\VM\ra\VN$ between the context categories of two von Neumann algebras and isomorphisms $\pair{\tilde f}{\cG_f}:\SigN\ra\SigN$ between their spectral presheaves (in the opposite direction). The latter isomorphisms are in a particular category of presheaves over varying base categories, see \cite{Doe12b} and definitions below. The result implies that there is a (contravariant) bijective correspondence between Jordan $*$-isomorphisms $f:\cM\ra\cN$ and isomorphisms $\pair{\tilde f}{\cG_f}:\SigN\ra\SigN$. The spectral presheaf determines a von Neumann algebra up to Jordan $*$-isomorphism, too.

Hence, neither the context category $\VM$ nor the spectral presheaf $\SigM$ are complete invariants. In this article, we will first show that the context category in fact determines a von Neumann algebra as a weakly closed Jordan algebra, i.e.,  as a $JBW$-algebra (Thm. \ref{Thm_OrderIsosGiveJBWIsos}). Then, drawing on the theory of orientations on $JB$-algebras and $JBW$-algebras \cite{Con74,AlfShu01,AlfShu03}, we will introduce the oriented context category and the oriented spectral presheaf and show that both are complete invariants of a von Neumann algebra not isomorphic to $\bbC\oplus\bbC$ and with no type $I_2$ summand (Thm. \ref{Thm_OrientedVMIsCompleteInvariant} respectively Thm. \ref{Thm_OrientedSigIsCompleteInvariant}).

\subsection{Definitions and preliminaries}
\begin{definition}
Let $\cM$ be a von Neumann algebra, and let $\VM$ be the set of unital commutative von Neumann subalgebras that have the same unit element as $\cM$. $\VM$, equipped with the partial order given by inclusion, is called the \textbf{context category of $\cM$}. The elements of $\VM$ are called \textbf{contexts}.
\end{definition}

\begin{definition}
The \textbf{spectral presheaf $\SigM$} of a von Neumann algebra $\cM$ is the presheaf over $\cM$ given
\begin{itemize}
	\item [(a)] on objects: for all $V\in\Ob{\VM}$, $\SigM_V:=\Sigma(V)$, the Gelfand spectrum of $V$,
	\item [(b)] on arrows: for all inclusions $i_{V'V}:V'\hra V$, $\SigM(i_{V'V}):\SigM_V\ra\SigM_{V'}$, $\ld\mt\ld|_{V'}$.
\end{itemize}
\end{definition}

\begin{definition}
A \textbf{$JB$-algebra $A$} is a real Jordan algebra which is also a Banach space such that
\begin{itemize}
	\item [(i)] $|a\cdot b|\leq|a||b|$,
	\item [(ii)] $|a^2|=|a|^2$,
	\item [(iii)] $|a^2|\leq|a^2+b^2|$.
\end{itemize}
\end{definition}

We will only consider unital $JB$-algebras. The canonical example of a $JB$-algebra is $(\cA_{\sa},\cdot)$, the self-adjoint part of a unital $C^*$-algebra $\cA$, equipped with the symmetrised product $a\cdot b=\frac{1}{2}(ab+ba)$. There is a category $\JB$ of unital $JB$-algebras, with unital norm-continuous Jordan homomorphisms as arrows.

\begin{definition}
A \textbf{$JBW$-algebra} is a unital $JB$-algebra that is the dual of a Banach space. 
\end{definition}

The canonical example is $(\cM_{\sa},\cdot)$ for a von Neumann algebra $\cM$. We say that $(\cM_{\sa},\cdot)$ is \textbf{the $JBW$-algebra associated with $\cM$}. There is a category $\JBW$ of $JBW$-algebras, with unital and normal, i.e., ultraweakly continuous Jordan homomorphisms as arrows.

$JB$-algebras and $JBW$-algebras are real algebras. They can be complexified, and the Jordan product can be extended canonically to the complexification $A+iA$. The complexification comes equipped with the involution $^*:A+iA\ra A+iA$, $a_1+ia_2\mt a_1-ia_2$. An arrow $f:A\ra B$ in the category $\JB$ can be extended uniquely to a unital, norm-continuous Jordan $*$-homomorphism $f:A+iA\ra B+iB$ between the complexifications. Similarly, an arrow $f:A\ra B$ in the category $\JBW$ can be extended uniquely to a unital, normal Jordan $*$-homomorphism between the complexifications.

\begin{definition}			\label{Def_CatOfPresheaves}
Let $\cD$ be a category. The category $\Pre\cD$ of $\cD$-valued presheaves has as its objects functors of the form $\ps P:\cJ\ra\cD^{\op}$, where $\cJ$ is a small category. Arrows are pairs
\begin{equation}
			\pair{H}{\io}: (\ps{\tilde P}:\tilde\cJ\ra\cD^{\op})\lra(\ps P:\cJ\ra\cD^{\op}),
\end{equation}
where $H:\cJ\ra\tilde\cJ$ is a functor and $\io:H^*\ps{\tilde P}\ra\ps P$ is a natural transformation in $(\cD^{\op})^{\cJ}$. Here, $H^*\ps{\tilde P}$ is the presheaf over $\cJ$ given by
\begin{equation}
			\forall J\in\cJ: H^*\ps{\tilde P}_J=\ps{\tilde P}_{H(J)}.
\end{equation}
Let $\ps P_i:\cJ_i\ra\cD^{\op}$, $i=1,2,3$, be three presheaves over different base categories. Given two composable arrows $\pair{H'}{\io'}:\ps P_3\ra\ps P_2$ and $\pair{H}{\io}:\ps P_2\ra\ps P_1$, the composite is $\pair{H'\circ H}{\io\circ\io'}:\ps P_3\ra\ps P_1$, where, for all $J\in\cJ_1$, the natural transformation $\io\circ\io'$ has components
\begin{equation}
			(\io\circ\io')_J=\io_J\circ\io'_{H(J)}:((H'\circ H)^*\ps P_3)_J=(\ps P_3)_{H'(H(J))} \ra (\ps P_2)_{H(J)} \ra (\ps P_1)_J.
\end{equation}
\end{definition}

Let $\HStone$ be the category of hyperstonean spaces and open continuous maps. The Gelfand spectrum of a commutative von Neumann algebra is a hyperstonean space, so the spectral presheaf of a von Neumann algebra is an object in the presheaf category $\Pre{\HStone}$.

\section{Contexts and Jordan structure}
As mentioned in the introduction, in \cite{DoeHar10} it was shown:
\begin{theorem}			\label{Thm_OrderIsosGiveJordanIsos}
Suppose $\cM,\cN$ are von Neumann algebras not isomorphic to $\bbC\oplus\bbC$ and without type $I_2$ summands. There is a bijective correspondence between order-isomorphisms $\tilde f:\VM\ra\VN$ and Jordan $*$-isomorphisms $f:\cM\ra\cN$ such that $\tilde f(V)$ is equal to the image $f[V]$ for all $V\in\VM$.
\end{theorem}
Hence, the context category $\VM$ determines the von Neumann algebra $\cM$ up to Jordan $*$-isomorphism. 

The proof of the above theorem shows that there is a bijective correspondence between order-isomorphisms $\tilde f:\VM\ra\VN$ and isomorphisms of orthomodular lattices (OMLs) $\ol f:\PM\ra\PN$. Since $\PM,\PN$ are complete OMLs, $\ol f$ is in fact an isomorphism of complete OMLs. By a theorem by Dye \cite{Dye55}, $\ol f$ can be extended to a Jordan $*$-isomorphism $f:\VM\ra\VN$ such that $f(p)=\ol f(p)$ for all $p\in\PM$. Since $\ol f$ is an isomorphism of complete OMLs, it will preserve increasing nets of projections, so $f$ is a normal, i.e., ultraweakly continuous map. Taking topology into account, we obtain a slight strengthening of Thm. \ref{Thm_OrderIsosGiveJordanIsos}:

\begin{theorem}			\label{Thm_OrderIsosGiveJBWIsos}
Suppose $\cM,\cN$ are von Neumann algebras not isomorphic to $\bbC\oplus\bbC$ and without type $I_2$ summands. There is a bijective correspondence between order-isomorphisms $\tilde f:\VM\ra\VN$ and normal Jordan $*$-isomorphisms $f:(\cM_{\sa},\cdot)\ra(\cN_{\sa},\cdot)$ of real $JBW$-algebras, i.e., isomorphisms in $\JBW$. Moreover, $f$ can be extended uniquely to a normal Jordan $*$-isomorphism between the complex weakly closed Jordan $*$-algebras $(\cM,\cdot)$ and $(\cN,\cdot)$ such that $\tilde f(V)$ is equal to the image $f[V]$ for all $V\in\VM$.
\end{theorem}
Hence, the context category $\VM$ determines $\cM$ as a (complex) weakly closed Jordan $*$-algebra.

\section{Commutators and skew order derivations}
In order to obtain a complete invariant that determines a von Neumann algebra $\cM$ up to isomorphism, we need some extra information beyond the spectral presheaf or the context category. The noncommutative product in $\cM$ is not determined by the Jordan product. In particular, any von Neumann algebra $\cM$ (with product $(a,b)\mt ab$) and its opposite algebra $\cM^{\op}$ (with product $(a,b)\mt ba$) induce the same weakly closed Jordan $*$-algebra $(\cM,\cdot)$. There are von Neumann algebras that are not isomorphic to their opposite algebras, a deep result shown by Connes by providing a specific factor that is not anti-isomorphic to itself \cite{Con75}.

The symmetrised, i.e, Jordan product $a\cdot b=\frac{1}{2}(ab+ba)$ and the antisymmetrised, i.e., Lie product $\frac{1}{2}[a,b]=\frac{1}{2}(ab-ba)$ together determine the noncommutative product $ab$, since
\begin{align}
			ab= a \cdot b + \frac{1}{2}[a,b].
\end{align}

\begin{lemma}
Let $f:\cM\ra\cN$ be a normal unital Jordan $*$-homomorphism between von Neumann algebras. $f$ is a (normal unital) homomorphism of von Neumann algebras, i.e., preserves the associative, noncommutative product, if and only if $f$ preserves commutators, that is,
\begin{align}
			\forall a,b\in\cM: f([a,b]) = [f(a),f(b)].
\end{align}
If $f$ is a normal Jordan $*$-isomorphism preserving commutators, then $f$ is an isomorphism of von Neumann algebras.
\end{lemma}

\begin{proof}
It is clear that if $f$ is a normal unital homomorphism of von Neumann algebras, then $f$ preserves commutators. Conversely, assume $f$ preserves commutators, then
\begin{align}
			f(ab) &= f(a\cdot b+\frac{1}{2}[a,b]) = f(a\cdot b)+\frac{1}{2}f([a,b])\\
			&= \frac{1}{2}(f(a)f(b)+f(b)f(a)+f(a)f(b)-f(b)f(a))\\
			&= f(a)f(b).
\end{align}
The statement about isomorphisms is clear.
\end{proof}

Since commutators are not given at the level of weakly closed Jordan $*$-algebras, we have to encode the information about the preservation of commutators in a different manner. In the following, we will draw heavily on the work by Connes \cite{Con74} and in particular Alfsen, and Shultz on orientations on operator algebras, see \cite{AlfShu98} and the books \cite{AlfShu01,AlfShu03} for details. 

\begin{definition}
An \textbf{order derivation $\de$} on a $JB$-algebra $A$ is a bounded linear operator such that $e^{t\de}(A^+)\subseteq A^+$ for all $t\in\bbR$, that is, $t\mt e^{t\de}$ is a one-parameter group of order automorphisms.

An order derivation $\de$ is \textbf{self-adjoint} if $\de=\de_{a}$ for some $a\in A$, where
\begin{align}
			\de_a:A &\lra A\\			\nonumber
			b &\lmt a\cdot b.
\end{align}
An order derivation $\de$ is \textbf{skew-adjoint} if $\de(1)=0$. The set of all order derivations on a $JB$-algebra $A$ is denoted $OD(A)$, and the set of skew order derivations is denoted $OD_s(A)$.
\end{definition}
Every order derivation $\de$ can be decomposed uniquely as the sum of a self-adjoint and a skew-adjoint order derivation.
\begin{lemma}			\label{Lemma_ede_iaGivesJordanAutoms}
(Alfsen/Shultz \cite{AlfShu98}) Let $A$ be a unital $JB$-algebra, and let $\de:A\ra A$ be an order derivation. The following are equivalent:
\begin{itemize}
	\item [(a)] $\de$ is skew,
	\item [(b)] $e^{t\de}$ is a Jordan automorphism for all $t\in\bbR$,
	\item [(c)] $\de$ is a derivation for the Jordan product.
\end{itemize}
\end{lemma}
Let $A=\cA_{\sa}$ be the self-adjoint part of a unital $C^*$-algebra. For each $a\in\cA$ (not necessarily self-adjoint), define a linear operator
\begin{align}			\label{Def_OrderDerivsOnVNAs}
			\de_{a}:A &\lra A\\			\nonumber
			b &\lmt \frac{1}{2}(ab+ba^*).
\end{align}
Note that if $a\in\cA_{sa}$, then $\de_a(b)=a\circ b$, and $\de_{ia}(b)=\frac{1}{2}(iab-iba)=\frac{i}{2}[a,b]$.
\begin{proposition}
(Alfsen/Shultz \cite{AlfShu98}) If $A=(\cM_{\sa},\cdot)$ is the $JBW$-algebra associated with a self-adjoint part of a von Neumann algebra $\cM$, then the order derivations on $A$ are the operators $\de_{a}$ defined above (where $a\in\cN$). An order derivation is self-adjoint if and only if $\de=\de_{a}$ for some self-adjoint $a\in\cN_{\sa}$, and is skew-adjoint if and only if $\de=\de_{ia}=\frac{i}{2}[a,-]$ for some $a\in\cN_{\sa}$.
\end{proposition}
This shows that the skew-adjoint order derivations on (self-adjoint parts of) von Neumann algebras encode commutators.
\begin{proposition}			\label{Prop_fVNAMorphismIfCommutesWithde_ia}
A normal unital Jordan homomorphism $f:\cM_{\sa}\ra\cN_{\sa}$ between the real parts of two von Neumann algebras extends to a (normal unital) homomorphism $f:\cM\ra\cN$ of von Neumann algebras if and only if
\begin{align}			\label{Cond_fCommutesWithde_ia}
			\forall a\in\cM_{\sa}: f\circ\de_{ia} = \de_{if(a)}\circ f.
\end{align}
If $f$ is a Jordan isomorphism such that this condition holds, then $f$ extends to an isomorphism of von Neumann algebras.
\end{proposition}

\begin{proof}
If $f$ is a normal unital Jordan homomorphism such that \eq{Cond_fCommutesWithde_ia} holds, then we have
\begin{align}
			\forall a,b\in\cM_{\sa}: f(\frac{i}{2}[a,b])=(f\circ\de_{ia})(b)=(\de_{if(a)}\circ f)(b)=\frac{i}{2}[f(a),f(b)],
\end{align}
so $f$ preserves all commutators between self-adjoint operators. Since any operator $c\in\cM$ can be decomposed uniquely as $c=a_1+ia_2$ for self-adjoint $a_1,a_2\in\cM_{\sa}$, it follows easily that $f$ preserves all commutators, and hence is a normal unital homomorphism $f:\cM\ra\cN$ of von Neumann algebras.

Conversely, if $f:\cM\ra\cN$ is a homomorphism of von Neumann algebras, then its restriction to $\cM_{\sa}$ is a Jordan homomorphism to $\cN_{\sa}$ such that condition \eq{Cond_fCommutesWithde_ia} holds. The statement about isomorphisms is obvious.
\end{proof}

Let $\Aut_{\JBW}(\cM_{\sa})$ denote the group of normal Jordan automorphisms, i.e., automorphisms in the category $\JBW$, of $(\cM_{\sa},\cdot)$. By Lemma \ref{Lemma_ede_iaGivesJordanAutoms}, each skew order derivation $\de_{ia}$ ($a\in\cM_{\sa}$) generates a one-parameter group of Jordan automorphisms
\begin{align}
			\tF: \bbR &\lra \Aut_{\JBW}(\cM_{sa})\\			\nonumber
			t &\lmt e^{t\de_{ia}}.
\end{align}
Conversely, if $e^{t\de}$ is a one-parameter group of Jordan automorphisms, then $\de$ is skew.

Using the expansion of $e^{t\de_{ia}}$ (and $e^{t\de_{if(a)}}$) into an exponential series, it is easy to check that, for all $a\in\cM_{\sa}$,
\begin{align}
			f\circ\de_{ia} = \de_{if(a)}\circ f\quad\text{ iff }\quad\forall t\in\bbR: f\circ e^{t\de_{ia}}=e^{t\de_{if(a)}}\circ f.
\end{align}
Hence, we obtain
\begin{corollary}			\label{Cor_fvNaHomomIfffCommutesWithOneParamGroupsOfJordanAutoms}
A normal unital Jordan homomorphism $f:\cM_{\sa}\ra\cN_{\sa}$ between the real parts of two von Neumann algebras extends to a normal unital homomorphism $f:\cM\ra\cN$ of von Neumann algebras if and only if
\begin{align}			\label{Cond_fCommutesWithetde_ia}
			\forall a\in\cM_{\sa}\;\forall t\in\bbR: f\circ e^{t\de_{ia}}=e^{t\de_{if(a)}}\circ f.
\end{align}
If $f$ is a Jordan isomorphism such that this condition holds, then $f$ extends to an isomorphism of von Neumann algebras.
\end{corollary}
Note that for every $t$, the maps $f\circ e^{t\de_{ia}}$ and $e^{t\de_{if(a)}}\circ f$ are Jordan homomorphisms from $\cM_{\sa}$ to $\cN_{\sa}$. They are also normal, since the one-parameter groups of Jordan homomorphisms generated by skew order derivations are in fact one-parameter groups of \emph{inner} automorphisms:
\begin{lemma}			\label{Lemma_SkewOrderDerivsGiveInnerAutoms}
(Alfsen/Shultz) If $\de_{ia}:\cM_{\sa}\ra\cM_{\sa}$ is a skew order derivation and $(\tF(t)=e^{t\de_{i\hA}})_{t\in\bbR}$ is the corresponding one-parameter group of Jordan $*$-automorphisms of $\cM$, then
\begin{align}
			\de_{ia}=\frac{d}{dt}(\tF(t))|_{t=0}
\end{align}
and
\begin{align}			
			\forall t\in\bbR: \tF(t)=e^{t\de_{ia}}:\cM &\lra \cM\\
			b &\lmt e^{\frac{i}{2}ta}b e^{-\frac{i}{2}ta}=u_{\frac{t}{2}} b u_{-\frac{t}{2}}.
\end{align}
\end{lemma}

\section{Orientations}			\label{Sec_Orientations}
\subsection{Associative products, central projections, and dynamical correspondences}
In general, a given $JBW$-algebra $A$ may admit several associative, noncommutative products that make $A+iA$ into a von Neumann algebra and induce the given Jordan product.\footnote{We assume that a given $JBW$-algebra $A$ indeed is associated with some von Neumann algebra, since this is the situation we are interested in. Moreover, we assume that $A$ is not associative. It it was, it would be a commutative von Neumann algebra already.} The definition of order derivations \eq{Def_OrderDerivsOnVNAs} depends on the chosen associative product. In particular, the definition of skew order derivations $\de_{ia}=\frac{i}{2}[a,-]$ ($a\in\cM_{\sa}$) on the $JBW$-algebra underlying a von Neumann algebra will depend on the choice of associative product in $\cM$.

Let $(a,b)\mt ab$ and $(a,b)\mt a\star b$ be two associative products on (the linear space underlying) a von Neumann algebra $\cM$ that induce the same Jordan product, i.e., for all $a,b\in A$,
\begin{align}
			a\cdot b=\frac{1}{2}(ab+ba) = \frac{1}{2}(a\star b+b\star a).
\end{align}
Alfsen and Shultz showed (see Thm. 7.103 and Lemma 7.100 in \cite{AlfShu01}) that any two such products differ by a central projection $c$ which is $1$ on the abelian part of $\cM$ in the following sense:
\begin{align}
			\forall a,b\in\cM: a\star b = cab+(1-c)ba.
\end{align}
Alfsen and Shultz use a central symmetry $z$ that is $1$ on the abelian part, which is given by $z=2c-1$. In the actual formula relating the products, the central projection $c$ shows up, so we will work with this. 

Hence, there is a bijective correspondence between central projections $c$ that are $1$ on the abelian part and associative products on (the linear space underlying) $\cM$ that induce the same Jordan product.

What this means is easiest to illustrate for the case of a factor (not of type $I_1$). If $\cM$ is a factor, there are two associative products on $\cM$ that induce the same Jordan product $(a,b)\mt \frac{1}{2}(ab+ba)$: one can either pick $(a,b)\mt ab$ or $(a,b)\mt a\star b:=ba$. There are two central projections in $\cM$ (and there is no abelian part), $c=0$ and $c=1$. The two products relate to each other by
\begin{align}
			\forall a,b\in\cM: a\star b = ba = 0ab+ (1-0)ba.
\end{align}
Considering skew order derivations, we have either
\begin{align}
			\forall a,b\in\cM_{\sa}: \de_{ia}(b)=\frac{i}{2}(ab-ba)=\frac{i}{2}[a,b]
\end{align}
for the product $(a,b)\mt ab$, or
\begin{align}
			\forall a,b\in\cM_{\sa}: \de_{ia;\star}=\frac{i}{2}(a\star b-b\star a)=\frac{i}{2}[b,a]
\end{align}
for the product $(a,b)\mt a\star b=ba$. Of course, we have
\begin{align}
			ab &= a\cdot b -i\de_{ia}(b) = \frac{1}{2}(ab+ba)+\frac{1}{2}(ab-ba),\\
			ba &= a\cdot b -i\de_{ia;\star}(b) = \frac{1}{2}(ab+ba)+\frac{1}{2}(ba-ab).
\end{align}
If $\cM$ is a finite direct sum of factors, we can pick the order of the product in each direct summand separately. The only exception is if there is a direct summand of type $I_1$, which is an abelian part where picking an order of the product of course is meaningless. One simply demands that on such a direct summand, no choice can be made. In all other direct summands (i.e., factors), there are two possible choices of order of the product. These choices can be encoded by picking a central projection $c$ (which means picking $0$ or $1$ in each factor) that is $1$ on each direct summand of type $I_1$.

In terms of commutators, in each factor $\cN$ in the decomposition of $\cM$ we pick either $\de_{ia}=\frac{1}{2}[a,-]$ or $\de_{ia;\star}=\frac{1}{2}[-,a]$ for all $a\in\cN_{\sa}$.

Every von Neumann algebra $\cM$ can be decomposed as a direct integral of factors, so this argument generalises. A central projection $c$ that is $1$ on the abelian part picks the order of the product in each factor showing up in the decomposition and hence fixes the product globally in $\cM$. 

A closely related way of thinking about this is by axiomatising the map $a\mt\de_{ia}$ from the elements of a real $JBW$-algebra $A$ to the skew order derivations on $A$.

\begin{definition}
(Def. 6.10 in \cite{AlfShu03}) A \textbf{dynamical correspondence} on a $JB$-algebra $A$ is a linear map
\begin{align}
	\psi: A &\lra OD_s(A)\\			\nonumber
	a &\lmt \psi_a
\end{align}
from $A$ into the set of skew order derivations on $A$ which satisfies the requirements
\begin{itemize}
	\item [(i)] $[\psi_a,\psi_b]=-[\de_a,\de_b]$ for all $a,b\in A$,
	\item [(ii)] $\psi_a a=0$ for all $a\in A$.
\end{itemize}
A dynamical correspondence on a $JB$-algebra $A$ will be called \textbf{complete} if it maps $A$ onto the set of all skew order derivations on $A$.
\end{definition}

\begin{theorem}
(Alfsen/Shultz) A $JBW$-algebra $A$ is (Jordan isomorphic to) the self-adjoint part of a von Neumann algebra if and only if there exists a dynamical correspondence on $A$. In this case, there is a bijective correspondence of associative products on $A+iA$ and dynamical correspondences on $A$. The dynamical correspondence on $A$ associated with an associative product $(a,b)\mt ab$ is
\begin{align}
			\psi_a b = \frac{i}{2}[a,b].
\end{align}
This dynamical correspondence is complete. Conversely, the associative product on $A+iA$ associated with a dynamical correspondence $\psi$ on $A$ is the complex bilinear extension of the product defined on $A$ by
\begin{align}
			ab = a\cdot b-i\psi_a b.
\end{align}
\end{theorem}
For a proof, see Prop. 6.11, Thm. 6.15 and Cor. 6.16 in \cite{AlfShu03}. Hence, picking a dynamical correspondence on a $JBW$-algebra is equivalent to picking commutators and thus an associative product.

Alfsen and Shultz moreover show (\cite{AlfShu03}, Thm. 6.18) that there is a bijective correspondence between dynamical correspondences on a $JBW$-algebra and orientations in the sense of Connes \cite{Con74}. In the following, we will speak of picking an \textbf{orientation} on a $JBW$-algebra $A$ when we mean picking a dynamical correspondence/commutators/an associative product.

\subsection{Relation to our results so far}
Prop. \ref{Prop_fVNAMorphismIfCommutesWithde_ia} shows that a normal unital Jordan homomorphism $f:\cM_{\sa}\ra\cN_{\sa}$ between the real parts of two von Neumann algebras extends to a (normal unital) homomorphism $f:\cM\ra\cN$ of von Neumann algebras if and only if
\begin{align}			\label{Cond2}
			\forall a\in\cM_{\sa}: f\circ\de_{ia} = \de_{if(a)}\circ f.
\end{align}
Here, a choice of dynamical correspondences both on $A=(\cM_{\sa},\cdot)$ and $B=(\cN_{\sa},\cdot)$ and hence a choice of orientations/commutators/associative products has been made. In this light, condition \eq{Cond2} can be interpreted as the requirement that the chosen orientations are preserved by $f$. 

Hence, the proposition states that $f$ is a morphism of von Neumann algebras if and only if it \emph{preserves the chosen orientations}. 

In terms of the one-parameter groups of inner (Jordan $*$-)automorphisms $t\mt e^{t\de_{ia}}$, the choice of an orientation amounts to picking the \emph{direction of time}. Consider the case of a factor $\cM$ first: if we pick $a\mt\de_{ia}=\frac{i}{2}[a,-]$ for each $a\in\cM$, then the associated one-parameter group of inner automorphisms is $t\mt e^{t\de_{ia}}$, but if we pick $a\mt\de_{ia;\star}=\frac{i}{2}[-,a]=-\frac{i}{2}[a,-]$, then the associated one-parameter group of inner automorphisms is $t\mt e^{-t\de_{ia}}$. In general, in the $JBW$-algebra $A=(\cM_{\sa},\cdot)$ underlying a von Neumann algebra $\cM$, such a choice of time direction has to be made for each factor in the decomposition of $\cM$ into factors, except for the abelian part of $\cM$, where the automorphisms act trivially. Such a global choice of time direction will be called a \textbf{time orientation on $A$}. 

By construction, the choice of a time orientation is equivalent to a choice of orientation/associative product on $A$, making it into a von Neumann algebra $\cM$. Different choices of time orientation on $A$ will give different associative products and hence different, potentially non-isomorphic von Neumann algebras $\cM,\cM',...$, but each associative product will induce the same Jordan product on $A$. 

Corollary \ref{Cor_fvNaHomomIfffCommutesWithOneParamGroupsOfJordanAutoms} shows that a normal unital Jordan homomorphism $f:\cM_{\sa}\ra\cN_{\sa}$ between the real parts of two von Neumann algebras extends to a normal unital homomorphism $f:\cM\ra\cN$ of von Neumann algebras if and only if
\begin{align}			\label{Cond3}
			\forall a\in\cM_{\sa}\;\forall t\in\bbR: f\circ e^{t\de_{ia}}=e^{t\de_{if(a)}}\circ f.
\end{align}
Here, a choice of time orientations both on $A=(\cM_{\sa},\cdot)$ and $B=(\cN_{\sa},\cdot)$ and hence a choice of orientations/commutators/associative products has already been made. Condition \eq{Cond3} can be interpreted as the requirement that the chosen time orientations are preserved by $f$. Hence, $f$ is a morphism of von Neumann algebras if and only if \emph{$f$ preserves the chosen time orientations}.

\section{The oriented context category as a complete invariant of a von Neumann algebra}
We now construct new complete invariants for von Neumann algebras. First, we focus on the context category $\VM$ of a von Neumann algebra $\cM$. Let $A=(\cM_{\sa},\cdot)$ be the $JBW$-algebra associated with $\cM$, and let
\begin{align}
			\psi: A &\lra OD_s(A)\\			\nonumber
			a &\lmt \de_{ia}=\frac{i}{2}[a,-]
\end{align}
be the dynamical correspondence on $A$ induced by the associative product of $\cM$. By the results by Alfsen and Shultz, $(A,\psi)$ is a complete invariant of $\cM$. Let $\VM$ be the context category of $\cM$. A von Neumann algebra $\cN$ is commutative if and only if the associated Jordan algebra $(\cN_{\sa},\cdot)$ is associative. This implies that $\VM$ can also be described in terms of the $JBW$-algebra $A=(\cM_{\sa},\cdot)$ alone: it is the set of weakly closed associative Jordan subalgebras of $A$, partially ordered by inclusion.

The dynamical correspondence $\psi$ determines the time direction of the one-parameter groups $t\mt e^{t\de_{ia}}$ of inner (Jordan) automorphisms of $A$. Moreover, each one-parameter group $t\mt e^{t\de_{ia}}$ induces a one-parameter group of order-automorphisms of $\VM$, where, for each $t\in\bbR$,
\begin{align}
			\widetilde{e^{t\de_{ia}}} : \VM &\lra \VM\\
			V &\lmt e^{\frac{i}{2}ta}V e^{-\frac{i}{2}ta}.
\end{align}
The time direction of $t\mt\widetilde{e^{t\de_{ia}}}$ is fixed by $\psi$. Conversely, each one-parameter group of order-isomorphisms of $\VM$ induces a one-parameter group of Jordan isomorphisms of $A$ by Thm. \ref{Thm_OrderIsosGiveJBWIsos}. This shows that picking the dynamical correspondence $\psi$ on $A$ is equivalent to picking the time direction of the one-parameter groups $t\mt\widetilde{e^{t\de_{ia}}}$ ($a\in A$) of order-automorphisms of $A$.

\begin{definition}
Let $\cM$ be a von Neumann algebra, and let $\VM$ be its context category. Let $\Aut_{ord}(\VM)$ denote the group of order-automorphisms of $\VM$. The map
\begin{align}
			\tilde\psi: \cM_{\sa}\times\bbR &\lra \Aut_{ord}(\VM)\\
			(a,t) &\lmt \widetilde{e^{t\de_{ia}}}
\end{align}
is called the \textbf{time orientation on order-automorphisms of $\VM$ induced by $\cM$}. When $\VM$ is equipped with this time orientation, is it called the \textbf{oriented context category of $\cM$}.
\end{definition}

\begin{theorem}			\label{Thm_OrientedVMIsCompleteInvariant}
Let $\cM,\cN$ be von Neumann algebras not isomorphic to $\bbC\oplus\bbC$ and with no type $I_2$ summands. There is a bijective correspondence between isomorphisms $f:\cM\ra\cN$ of von Neumann algebras and order-isomorphisms $\tilde f:\VM\ra\VN$ of the context categories that preserve the orientations on $\VM$ and $\VN$ induced by $\cM$ and $\cN$, respectively, that is, order-isomorphisms $\tilde f:\VM\ra\VN$ such that the diagram
\begin{equation}			\label{Diag_tildefPreservesOrientations}
			\xymatrix{
			\VM  \ar[rr]^{\tilde f} \ar[dd]_{\widetilde{e^{t\de_{ia}}}} & & \VN \ar[dd]^{\widetilde{e^{t\de_{if(a)}}}}
			\\ & & \\ 
			\VM \ar[rr]^{\tilde f} & & \VN.
			}
\end{equation}
commutes for all $a\in\cM_{\sa}$ and all $t\in\bbR$.
\end{theorem}

\begin{proof}
Let $\tilde f:\VM\ra\VN$ be an order-isomorphism. By Thm. \ref{Thm_OrderIsosGiveJBWIsos}, $\tilde f$ corresponds to a unique Jordan $*$-isomorphism $f:\cM\ra\cN$. Moreover, $\tilde f:\VM\ra\VN$ preserves orientations (i.e., the diagram \eq{Diag_tildefPreservesOrientations} commutes for all $a\in\cM_{\sa}$ and all $t\in\bbR$) if and only if
\begin{equation}
			\xymatrix{
			\cM \ar[rr]^{f} \ar[dd]_{e^{t\de_{ia}}} & & \cN \ar[dd]^{e^{t\de_{if(a)}}}
			\\ & & \\ 
			\cM \ar[rr]^{f} & & \cN.
			}
\end{equation}
commutes for all $a\in\cM_{\sa}$ and all $t\in\bbR$. By Cor. \ref{Cor_fvNaHomomIfffCommutesWithOneParamGroupsOfJordanAutoms}, this means that $f:\cM\ra\cN$ is an isomorphism of von Neumann algebras.
\end{proof}

Hence, the context category $\VM$, together with the time orientation induced by $\cM$, is a complete invariant of $\cM$.

\section{The oriented spectral presheaf as a complete invariant of a von Neumann algebra}
The following result was proven in \cite{Doe12b}:\footnote{In this reference, the category $\Pre{\KHaus}$ of presheaves valued in compact Hausdorff spaces is used. Morphisms in $\KHaus$ are continuous maps, not necessarily open, while morphisms in $\HStone$ are continuous and open, but isomorphisms in $\KHaus$ (i.e., homeomorphisms) of course are continuous and open. Hence, one can use $\Pre{\KHaus}$ as well.}
\begin{theorem}			\label{Thm_JordanIsosGiveSpecPreshIsosAndViceVersa}
Let $\cM,\cN$ be von Neumann algebras not isomorphic to $\bbC\oplus\bbC$ and with no type $I_2$ summands. There is a bijective correspondence between normal Jordan isomorphisms $f:\cM_{\sa}\ra\cN_{\sa}$ and isomorphisms $F:\SigN\ra\SigM$ in $\Pre{\HStone}$ of the spectral presheaves of $\cM,\cN$.
\end{theorem}

The forward direction works more generally: a normal Jordan homomorphism $f:\cM_{\sa}\ra\cN_{\sa}$ induces a morphism $\pair{\tilde f}{\cG_f}:\SigN\ra\SigM$ of the spectral presheaves. Concretely, given a normal Jordan homomorphism $f:\cM_{\sa}\ra\cN_{\sa}$, we first extend it uniquely to a normal Jordan $*$-homomorphism $f:\cM\ra\cN$. This induces an order-preserving map
\begin{align}
			\tilde f:\VM &\lra \VN\\			\nonumber
			V &\lmt f(V).
\end{align}
By precomposition with $\tilde f$, one can `pull back' $\SigN$ to $\SigN\circ\tilde f$, which is a presheaf over $\VM$. For each $V\in\VM$, the restriction $f|_{V}:V\ra f(V)$ is a morphism of commutative von Neumann algebras, so by Gelfand duality there is a continuous map
\begin{align}
			\cG_{f;V}:\Sigma(f(V)))=(\SigN\circ\tilde f)_V &\lra \SigM_V = \Sigma(V)\\
			\ld &\lmt \ld\circ f|_V
\end{align}
between the Gelfand spectra. As was shown in \cite{Doe12b}, the $\cG_{f;V}$ ($V\in\VM$) are the components of a natural transformation $\cG_f:\SigN\circ\tilde f\ra\SigM$. We obtain a morphism
\begin{align}
			\pair{\tilde f}{\cG_f}:\SigN \overset{\tilde f}{\lmt} \SigB\circ\tilde f \overset{\cG_{f}}{\lmt} \SigM
\end{align}
of the spectral presheaves, which is a morphism in the category $\Pre{\HStone}$. For the converse direction, which works for isomorphisms, see \cite{Doe12b}.

Let $\Aut(\SigM)$ denote the automorphism group of $\SigM$ in the category $\Pre{\HStone}$, and recall that $\Aut_{\JBW}(\cM)$ is the automorphism group of $(\cM_{\sa},\cdot)$ in the category of $JBW$-algebras.
\begin{corollary}			\label{Cor_GroupIsos}
Let $\cM$ be a von Neumann algebra not isomorphic to $\bbC\oplus\bbC$ and with no type $I_2$ summand, then $\Aut(\SigM)$ is contravariantly isomorphic to $\Aut_{\JBW}(\cM_{\sa})$.
\end{corollary}

\begin{definition}
A representation
\begin{align}
			F: \bbR &\lra \Aut(\SigM)\\			\nonumber
			t &\lmt \pair{\Ga_t}{\eta_t}
\end{align}
of the additive group of real numbers by automorphisms of the spectral presheaf is called a \textbf{flow on the spectral presheaf $\Sig$}. Let $\tF:\bbR\ra\Aut_{\JBW}(\cM_{\sa})$ be the one-parameter group of normal Jordan automorphisms of $(\cM_{\sa},\cdot)$ corresponding to $F$ by Cor. \ref{Cor_GroupIsos}. The flow $F$ is called
\begin{itemize}
	\item [(a)] \textbf{weakly continuous} if $\tF$ is pointwise weakly continuous,
	\item [(b)] \textbf{continuous} if $\tF$ is pointwise norm-continuous,
	\item [(c)] \textbf{inner} if $\tF$ is of the form $(\tF(t)=e^{t\de})_{t\in\bbR}$ for some order derivation $\de$ on $\cN_{\sa}$.
\end{itemize}
\end{definition}
If $F:\bbR\ra\Aut(\SigM)$ is an inner flow, then the generating order derivation $\de$ must be skew by Lemma \ref{Lemma_ede_iaGivesJordanAutoms}, since $\tF(t)=e^{t\de}$ is a Jordan automorphism for every $t$. Skew order derivations generate one-parameter groups of inner automorphisms by Lemma \ref{Lemma_SkewOrderDerivsGiveInnerAutoms}, hence the name inner flow is justified. 

Flows on the spectral presheaf were first defined in \cite{Doe12c}, where it was shown that the time evolution of quantum systems both in the Heisenberg and the Schr\"odinger picture can naturally be described in terms of flows.

Let $f:\cM_{\sa}\ra\cN_{\sa}$ be a normal Jordan homomorphism, let $a\in\cM_{\sa}$, and let
\begin{align}
			\tilde G_{f;a}:\bbR &\lra \Hom_{\JBW}(\cM_{\sa},\cN_{\sa})\\	\nonumber
			t &\lmt f\circ e^{t\de_{ia}}
\end{align}
and
\begin{align}
			\tilde H_{f;a}:\bbR &\lra \Hom_{\JBW}(\cM_{\sa},\cN_{\sa})\\	\nonumber
			t &\lmt e^{t\de_{if(a)}}\circ f.
\end{align}
These are two one-parameter groups of normal Jordan homomorphisms. Cor. \ref{Cor_fvNaHomomIfffCommutesWithOneParamGroupsOfJordanAutoms} shows that the normal Jordan homomorphism $f:\cM_{\sa}\ra\cN_{\sa}$ can be extended to an homomorphism $f:\cM\ra\cN$ of von Neumann algebras if and only if for every $a\in\cM_{sa}$, we have
\begin{align}
			\tilde G_{f;a} = \tilde H_{f;a}.
\end{align}
The same holds true if we replace `homomorphism' by `isomorphism' in the statement. Let $G_{f;a}$ be the one-parameter group of isomorphisms in $\Pre{\HStone}$ from $\SigN$ to $\SigM$ corresponding to $\tilde G_{f;a}$ by Thm. \ref{Thm_JordanIsosGiveSpecPreshIsosAndViceVersa}, and let $H_{f;a}$ be the one-parameter group of isomorphisms corresponding to $\tilde H_{f;a}$. Concretely,
\begin{align}
			G_{f;a}:\bbR &\lra \Hom_{\Pre{\HStone}}(\SigN,\SigM)\\	\nonumber
			t &\lmt \pair{\widetilde{f\circ e^{t\de_{ia}}}}{\cG_{f\circ e^{t\de_{ia}}}}
\end{align}
and
\begin{align}
			H_{f;a}:\bbR &\lra \Hom_{\Pre{\HStone}}(\SigN,\SigM)\\ \nonumber
			t &\lmt \pair{\widetilde{e^{t\de_{if(a)}}\circ f}}{\cG_{e^{t\de_{if(a)}}\circ f}}.
\end{align}
We see that a normal Jordan isomorphism $f:\cM_{\sa}\ra\cN_{\sa}$ can be extended to a isomorphism $f:\cM\ra\cN$ if and only if for all $a\in\cM_{\sa}$, we have $G_{f;a}=H_{f;a}$. We want to interpret this condition geometrically.

By the rules of composition in the category $\Pre{\HStone}$, for all $t\in\bbR$,
\begin{align}
			\pair{\widetilde{f\circ e^{t\de_{ia}}}}{\cG_{f\circ e^{t\de_{ia}}}} 
			= \pair{\widetilde{e^{t\de_{ia}}}}{\cG_{e^{t\de_{ia}}}}\circ\pair{\tilde f}{\cG_f}
\end{align}
and
\begin{align}
			 \pair{\widetilde{e^{t\de_{if(a)}}\circ f}}{\cG_{e^{t\de_{if(a)}}\circ f}}
			= \pair{\tilde f}{\cG_f}\circ\pair{\widetilde{e^{t\de_{if(a)}}}}{\cG_{e^{t\de_{if(a)}}}}.
\end{align}
Here, $t\mt\pair{\widetilde{e^{t\de_{ia}}}}{\cG_{e^{t\de_{ia}}}}$ is the inner flow on $\SigM$ induced by the skew order derivation $\de_{ia}$ on $(\cM_{\sa},\cdot)$, and $t\mt\pair{\widetilde{e^{t\de_{if(a)}}}}{\cG_{e^{t\de_{if(a)}}}}$ is the inner flow on $\SigN$ induced by the skew order derivation $\de_{if(a)}$ on $(\cN_{\sa},\cdot)$. Hence, the condition $G_{f;a}=H_{f;a}$ for all $a\in\cM_{\sa}$ is equivalent to requiring that the diagram
\begin{equation}			\label{Diag_MorPreservesFlows1}
			\xymatrix{
			\SigN  \ar[rr]^{\pair{\tilde f}{\cG_f}} \ar[dd]_{\pair{\widetilde{e^{t\de_{if(a)}}}}{\cG_{e^{t\de_{if(a)}}}}} & & \SigM \ar[dd]^{\pair{\widetilde{e^{t\de_{ia}}}}{\cG_{e^{t\de_{ia}}}}}
			\\ & & \\ 
			\SigN \ar[rr]^{\pair{\tilde f}{\cG_f}} & & \SigM
			}
\end{equation}
commutes for all $a\in\cM_{\sa}$ and all $t\in\bbR$. 

We can reformulate the condition as follows: let $b\in\cN_{\sa}$, and let $a=f^{-1}(b)$. The latter exists and is unique, because $f$ is a (Jordan) isomorphism. If the diagram
\begin{equation}			\label{Diag_MorPreservesFlows2}
			\xymatrix{
			\SigN  \ar[rr]^{\pair{\tilde f}{\cG_f}} \ar[dd]_{\pair{\widetilde{e^{t\de_{ib}}}}{\cG_{e^{t\de_{ib}}}}} & & \SigM \ar[dd]^{\pair{\widetilde{e^{t\de_{ia}}}}{\cG_{e^{t\de_{ia}}}}}
			\\ & & \\ 
			\SigN \ar[rr]^{\pair{\tilde f}{\cG_f}} & & \SigM
			}
\end{equation}
commutes for all $b\in\cN_{sa}$, $a=f^{-1}(b)$, and all $t\in\bbR$, we say that \textbf{$\pair{\tilde f}{\cG_f}$ preserves inner flows}. This condition is clearly equivalent to the condition that the diagram of the form \eq{Diag_MorPreservesFlows1} commutes for all $a\in\cM_{\sa}$ and all $t\in\bbR$, and hence to the condition $G_{f;a}=H_{f;a}$ for all $a\in\cM_{\sa}$, which in turn is equivalent to $f$ being an isomorphism of von Neumann algebras.

As we saw, the associative product in a given von Neumann algebra $\cM$ induces a dynamical correspondence $a\mt\de_{ia}=\frac{i}{2}[a,-]$ ($a\in\cM_{\sa}$) on the $JBW$-algebra $(\cM_{\sa},\cdot)$. This determines a time orientation of the one-parameter groups of inner automorphisms, $t\mt e^{t\de_{ia}}$, and hence a time orientation of the inner flows on $\SigM$, $t\mt\pair{\widetilde{e^{t\de_{ia}}}}{\cG_{e^{t\de_{ia}}}}$. We say that this is \textbf{the orientation on $\SigM$ induced by $\cM$}.

If an isomorphism $\pair{\tilde f}{\cG_f}:\SigN\ra\SigM$ preserves inner flows, i.e., the diagram \eq{Diag_MorPreservesFlows2} commutes for all $b\in\cN_{\sa}$, $a=f^{-1}(b)$, and all $t\in\bbR$, we also say that $\pair{\tilde f}{\cG_f}$ \textbf{preserves the orientations} on $\SigN$ and $\SigN$ induced by $\cM$ and $\cN$, respectively. If $\tilde\cM$ is another von Neumann algebra with the same associated $JBW$-algebra $(\cM_{\sa},\cdot)$, then the product $(a,b)\mt a\star b$ in $\tilde\cM$ induces a different dynamical correspondence $a\mt\de_{ia;\star}=\frac{i}{2}[a,-]_\star$ ($a\in\cM_{\sa}$), and hence another time orientation of inner flows on $\SigM$, $t\mt\pair{\widetilde{e^{t\de_{ia;\star}}}}{\cG_{e^{t\de_{ia;\star}}}}$. 

Conversely, different time orientations of inner flows on $\SigM$ determine different associative products in (the linear space underlying) $\cM$.

\begin{theorem}			\label{Thm_OrientedSigIsCompleteInvariant}
Let $\cM,\cN$ be von Neumann algebras not isomorphic to $\bbC\oplus\bbC$ and with no type $I_2$ summands, and let $\SigM,\SigN$ be their spectral presheaves, equipped with the orientations induced by $\cM$ and $\cN$, respectively. There is a bijective correspondence between isomorphisms $f:\cM\ra\cN$ of von Neumann algebras and orientation-preserving isomorphisms $\pair{\tilde f}{\cG_f}:\SigN\ra\SigM$, that is, isomorphisms in $\Pre{\HStone}$ such that the diagram
\begin{equation}			\label{Diag_MorPreservesFlows3}
			\xymatrix{
			\SigN  \ar[rr]^{\pair{\tilde f}{\cG_f}} \ar[dd]_{\pair{\widetilde{e^{t\de_{ib}}}}{\cG_{e^{t\de_{ib}}}}} & & \SigM \ar[dd]^{\pair{\widetilde{e^{t\de_{ia}}}}{\cG_{e^{t\de_{ia}}}}}
			\\ & & \\ 
			\SigN \ar[rr]^{\pair{\tilde f}{\cG_f}} & & \SigM
			}
\end{equation}
commutes for all $b\in\cN_{sa}$, $a=f^{-1}(b)$, and all $t\in\bbR$.
\end{theorem}

\begin{proof}
By Thm. \ref{Thm_JordanIsosGiveSpecPreshIsosAndViceVersa}, every isomorphism $\pair{\tphi}{\eta}:\SigN\ra\SigN$ in $\Pre{\HStone}$ corresponds to a unique normal Jordan isomorphism $f:\cM_{\sa}\ra\cN_{\sa}$ and hence to a normal Jordan $*$-isomorphism $f:\cM\ra\cN$. We can write $\pair{\tphi}{\eta}=\pair{\tilde f}{\cG_f}$. A Jordan isomorphism $f$ is a von Neumann isomorphism if and only if $f$ preserves commutators, which by construction is equivalent to $\pair{\tilde f}{\cG_f}$ preserving orientations (in the sense that the diagram \eq{Diag_MorPreservesFlows2} commutes for all $b\in\cN_{\sa}$, $a=f^{-1}(b)$, and all $t\in\bbR$).
\end{proof}

This shows that the spectral presheaf $\SigM$, together with the orientation induced by $\cM$, is a complete invariant of a von Neumann algebra $\cM$. 

\bigskip

\textbf{Acknowledgements.} I thank Chris Isham, Rui Soares Barbosa, Pedro Resende and Jonathon Funk for discussions and Carmen Constantin for feedback. Moreover, I am grateful to Yuri Manin and the Max-Planck-Institute for Mathematics in Bonn for hospitality during a visit in March/April 2013, when some first steps leading tho the present article were taken.

\vspace{0.5cm}

\noindent\textsc{Andreas D\"oring, Institute of Theoretical Physics I,\\
Department of Physics, Universit\"at Erlangen-N\"urnberg,\\
Staudtstra\ss e 7, 91058 Erlangen, Germany}\\
\texttt{andreas.doering@fau.de}

\end{document}